\documentclass[12pt]{elsarticle}

\usepackage[margin=1in]{geometry}  
\usepackage{graphicx}              
\usepackage{amsmath}               
\usepackage{amsfonts}              
\usepackage{amsthm}                
\usepackage{mdwlist}               
\usepackage{float}
\usepackage{subfig}
\usepackage{array}
\usepackage{color}

\restylefloat{figure}

\newtheorem{thm}{Theorem}[section]

\newtheorem{prop}[thm]{Proposition}
\newtheorem{cor}[thm]{Corollary}

\newtheorem{defn}[thm]{Definition}

\newtheorem{remark}{Remark}

\newcommand{\ZZ}{\mathbb{Z}}      

\newcommand{\inimg}[2]{\begin{array}{c}\includegraphics[width=#2]{#1}\end{array}}

\newcommand{\iso}{\cong}

\newcommand{\bracket}[1]{\left\langle #1 \right \rangle}
\newcommand{\parens}[1]{\!\left( #1 \right)}

\journal{Topology and its Applications}
\title{Efficient Computation of the Kauffman Bracket}
\author[le]{Lauren Ellenberg}
\author[gn]{Gabriella Newman} 
\author[ss]{Stephen Sawin}
\author[js]{ Jonathan Shi}

\address[le]{Arcadia College, laellenberg@gmail.com}
\address[gn]{Carleton College, gabriellanewman@gmail.com}
\address[ss]{ Fairfield University,
  ssawin@fairfield.edu}
\address[js]{University of Washington, jshi@cs.washington.edu}

\begin{document}

\begin{keyword}
Kauffman Bracket \sep Jones Polynomial \sep knot theory \sep link invariants \sep tangles \sep girth \sep computational complexity 


\end{keyword}
\maketitle

\begin{abstract}
This paper bounds the computational cost of computing the Kauffman
bracket of a link in terms of the crossing number of that link.
Specifically, it is shown that the image of a tangle with $g$ boundary
points and $n$ crossings in the Kauffman
bracket skein module is a linear combination of $O(2^g)$ basis
elements, with each coefficient a polynomial with at most $n$ nonzero terms,
each with integer coefficients, and that the link can be built one
crossing at a time as a sequence of tangles with maximum number of
boundary points bounded by $C\sqrt{n}$ for some $C.$  From this it
follows that the computation of the Kauffman bracket of the link takes
time and memory a polynomial in $n$ times $2^{C\sqrt{n}}.$

\end{abstract}

\section{Introduction}

One of the principal advantages of the so-called quantum link invariants has always been their computability, in contrast to the intuitive but pragmatically uncomputable traditional link invariants, such as crossing number.  Indeed the simplest, the Kauffman bracket, can realistically be taught to a clever high school student.  It is thus a reasonable question what the computational complexity of the Kauffman bracket is. It is also a practical question.   For instance, there are conjectures, such as the claim that the Jones polynomial can detect the unknot, which could reasonably be tested empirically if the computation of the bracket or Jones polynomial were tractible for large knots.  Perhaps more interestingly, the computability of the Kauffman bracket is closely related to that of its cousin the Khovanov Homology.  In \cite{FGMW10}, Freedman, Gompf, Morrison and Walker computed it for large links in the hopes of finding counterexamples to the smooth $4$-dimensional Poincar\'e conjecture, and suggested larger links that it would be interesting to compute.   

The naive algorithm to compute the Kauffman bracket (smooth each crossing) is exponential in the number of crossings.  Although it is conceptually simple, this makes the computation effectively impossible for large links.  But in fact the local nature of the Kauffman bracket allows considerable savings from divide and conquer style approaches, and a heuristic argument that they reduce the cost to exponential in the square root of the number of crossings is not difficult.  For many years this claim had the status of a folk theorem.   The analogous statement for the Khovanov Homology was explictly conjectured by Bar-Natan in \cite{BarNatan02,BarNatan07}, who also sketched the algorithm. 

In this article the authors prove that the Kauffman bracket can be computed in time exponential in the square root of the number of crossings.  Because the computation happens on tangles, rather than links, several well known results about the nature of the Kauffman bracket on links must be generalized to tangles.  Specifically, in the expansion of a tangle $T$ as a linear combination of basis tangles with Laurent polynomials for coefficients, it is proved that the exponents occurring in each polynomial agree modulo $4,$ and that the span is bounded by four times the number of crossings. Also, bounds on the cutwidth of a planar graph are used to bound the girth of a link, giving the square root of the number of crossings which appears in the exponent.  

The authors would like to thank the National Science Foundation and Fairfield University for their support of the Fairfield University REU program, during the 2012 summer of which this work was produced.

\section{The Kauffman Skein Modules}

 \subsection{The Kauffman bracket}

\begin{defn}
A \textbf{framed link} (\cite{Turaev94}) is a link together with a nowhere zero vector field on
the range of the embedding, considered up to isotopy. A diagram of a
framed link is always presumed to have the vector field pointing up
(directly at the reader).
\end{defn}

\begin{thm}[Reidemeister \cite{BZ86,Trace83}] \label{reidemeister}
Two framed link diagrams represent the same framed link if and only if they can be 
connected by a sequence of the following three framed Reidemiester moves and 
their reverses.  Here the moves are understood to relate any two framed link 
diagrams that agree outside the dotted circle. 

R1: $\inimg{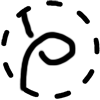}{1em} \leftrightarrow \inimg{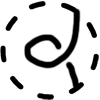}{1em}$

R2: $\inimg{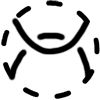}{1em} \leftrightarrow \inimg{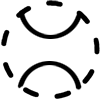}{1em}$

R3: $\inimg{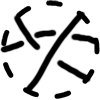}{1em} \leftrightarrow
\inimg{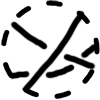}{1em}
$
\end{thm}

\begin{thm}[\cite{kauffman85}]\label{def:kauffbrack}
There is a unique assignment of a polynomial in
  $\ZZ[A,A^{-1}]$ to each framed link, called the Kauffman
  bracket $\langle L\rangle$ of the framed link $L,$  which sends the trivial link to
  $1$ and obeys the following two rules. 
These rules relate  the brackets of framed link diagrams that
  agree 
outside the dotted circle. 

\begin{equation} \label{eq:kauffman1}
\langle \inimg{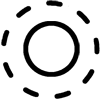}{1em}\rangle=(-A^2-A^{-2})\langle
\inimg{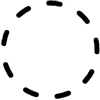}{1em}\rangle
\end{equation}
\begin{equation} \label{eq:kauffman2}
\langle \inimg{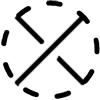}{1em}\rangle=A\langle \inimg{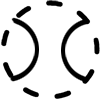}{1em}\rangle+A^{-1}\langle \inimg{2B1T.png}{1em}\rangle
\end{equation}

\end{thm}

\subsection{Skein modules}

\begin{defn}
 A \textbf{tangle diagram} $T$ is a finite multigraph embedded in  a disk with the following properties.  Every vertex is either quadravalent, bivalent or univalent.  Univalent vertices are only on the boundary, their unique edge is transverse to the boundary and the graph otherwise does not intersect the boundary.  Each bivalent vertex meets the same edge on both sides to form a simple closed loop. Each quadravalent vertex (called a crossing) comes equipped with a pair of opposite edges labeled as ``over'' and the other pair labeled ``under,'' the labeling represented visually exactly as are over and under strands in a crossing of a link diagram.  Tangle diagrams are considered up to    
 ambient isotopy of the interior of the disk.  
 A \textbf{(framed) tangle} is a tangle diagram considered   
 up to the (framed) Reidemeister moves. 
\end{defn}

In any link diagram, the interior of any disk whose boundary intersects the projection of the link transversely can be naturally identified with a tangle diagram, provided a bivalent vertex is chosen for each standardly embedded separated unknot component.  Replacing this tangle diagram by another which represents the same tangle gives a different projection of the same link by Theorem~\ref{reidemeister}.

\begin{defn} The boundary circle of a tangle $T,$ together with the univalent vertices on this boundary, is called the \textbf{type} of $ T.$ 
  Vertices on the boundary of the disk are called \textbf{boundary
    points.} The disk with the multigraph removed  is a union of connected faces, which are either  
\textbf{boundary faces} which touch the boundary, or \textbf{interior
   faces} which do not. 
A \textbf{base tangle diagram} of 
a given type
 is a crossingless, loopless tangle diagram (i.e., only univalent vertices). 
\end{defn}

For any  type $B,$ consider the
free module $\mathcal{D}_{B}$ over the ring $\ZZ[A,A^{-1}]$ generated by the set of all planar isotopy equivalence classes of framed 
tangle diagrams of the  given type, the free module $\mathcal{T}_{B}$ generated  by framed tangles of the given type, and the free module $\mathcal{B}_{B}$ generated by planar isotopy classes of base tangle diagrams. Equations like R1-R3 and Eqs.~\eqref{eq:kauffman1} and~\eqref{eq:kauffman2} relating tangle diagrams define elements of $\mathcal{D}_{B}$ or $\mathcal{T}_{B}$ by taking two or three tangle diagrams  that agree except in a small disk where they look as in the equation and taking the difference of the left and right sides of the equation.  Let $\mathcal{R}_{B}$ be the submodule of $\mathcal{D}_{B}$ generated by the Reidemeister moves, and $\mathcal{K}_{B}$ be the submodule of either $\mathcal{D}_{B}$ or $\mathcal{T}_{B}$ (by abuse of notation) generated by Eqs.~\eqref{eq:kauffman1} and~\eqref{eq:kauffman2}.  Then by definition $\mathcal{T}_{B} \iso \mathcal{D}_{B}/\mathcal{R}_{B}$ as $\ZZ[A,A^{-1}]$ modules, the standard proof that the Kauffman bracket is invariant under Reidemeister moves  shows 
\[\mathcal{R}_{B} \subset \mathcal{K}_{B}\]
and the standard argument that the Kauffman relations determine the Kauffman bracket up to an overall factor readily extends to show that the natural embedding $\mathcal{B}_{B}$ into $\mathcal{D}_{B}$ gives
\[\mathcal{B}_{B} \iso \mathcal{D}_{B}/\mathcal{K}_{B} \iso \mathcal{T}_{B}/ \mathcal{K}_{B}.\]
Call this free quotient the
\textbf{Kauffman skein module} $\mathcal{M}_{B}.$   Thus every framed tangle $T$ can be
written uniquely as a linear combination $\bracket{T}$ of basis tangles with the
same boundary, the coefficients being elements of $\ZZ[A,A^{-1}].$

\subsection{Locality}

The fundamental property of the Kauffman skein module that will be used repeatedly is its \emph{locality} \cite{BarNatan02}[Sec. 7].  Specifically, consider a tangle diagram $T$ of type $B$  and consider a disk inside it that intersects the diagram transversely, and thus determines a subtangle of type $B'.$  It is natural to imagine gluing any element of $\mathcal{T}_{B'}$ into $T,$ in the sense that an element of $\mathcal{T}_{B'}$ is a linear combination (with coefficients in $\ZZ[A,A^{-1}]$) of tangle diagrams of type $B',$ each of which can be glued into the disk to get a new tangle of type $B.$ The same linear combination of these new tangles gives an element of $\mathcal{T}_{B}.$  This is well defined because planar isotopy of the internal disk is a special case of planar isotopy of the larger disk.  Thus each tangle and disk intersecting it appropriately defines a map from $\mathcal{T}_{B'}$ into $\mathcal{T}_{B},$ and these maps all send $\mathcal{R}_{B'}$ to $\mathcal{R}_{B}$ and $\mathcal{K}_{B'}$ to $\mathcal{K}_{B}.$  This means that if $T'$ were the portion of $T$ within the inner disk, and we computed $\bracket{T'},$ we could glue that back in to get a linear combination of tangle diagrams which, as an element of $\mathcal{T}_{B}$ is mapped by the quotient to $\bracket{T} \in \mathcal{M}_B.$  Thus we can compute $\bracket{T}$ by first computing the Kauffman bracket of a subtangle.  This suggests the strategy for computing the Kauffman bracket of a tangle by choosing an increasing sequence of disks wisely and simplifying the tangle within the disks sequentially.  Of course, the Kauffman bracket  of a link can be naturally identified with the coefficient of the empty link in its image in the one-dimensional Kauffman skein module of the disk with no boundary points, so computing the Kauffman bracket of a link is a special case of this algorithm.

\section{Coefficient Exponents}

\subsection{Checkerboarding}
\begin{defn}
A \textbf{checkerboarded tangle} is a tangle in which each face has
been marked as either dark  or light, so that no dark face shares an edge with a different dark face and no light face shares an edge with a different light face.  Since any tangle can be embedded in a  link diagram, and since these can be checkerboarded, there exists a checkerboarding of every tangle, and in fact there are obviously two for each tangle [Figure \ref{fig:checktngl}].
The choice of color is determined by the choice on the
  type, and thus a checkerboarding of a tangle determines a
  checkerboarding of all tangles with the same type.
\begin{figure}[!ht]
\centering
  \subfloat[A Tangle]{\includegraphics[width=0.25\textwidth]{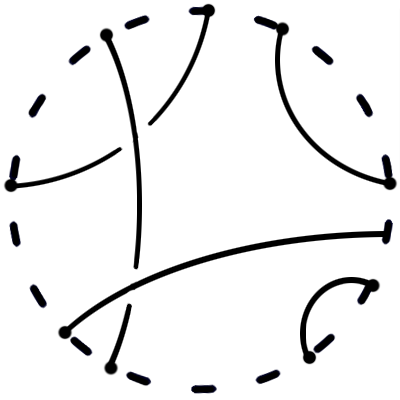}}
  ~ ~
  \subfloat[Checkerboarded ($e=2,$ $w=-2$)]{\includegraphics[width=0.25\textwidth]{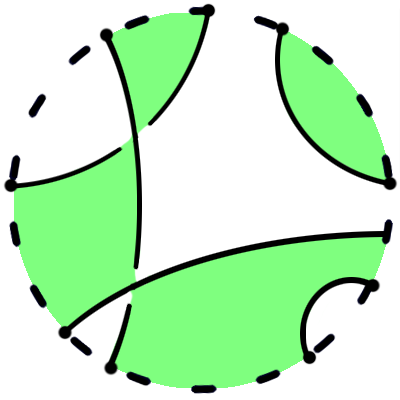}}
  ~~
  \subfloat[Alternate checkerboarding ($e=2,$ $w=2$)]{\includegraphics[width=0.25\textwidth]{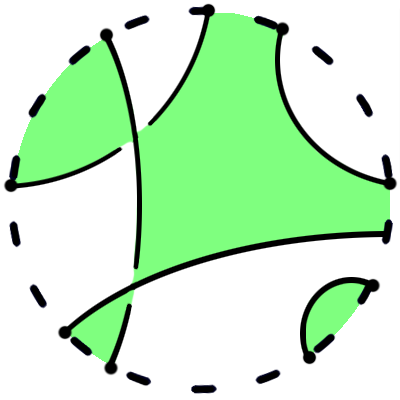}}
  \caption{A checkerboarded tangle}\label{fig:checktngl}
\end{figure}

\end{defn}
 
Of course the definition of tangle is meant to represent something three-dimensional without explicitly saying it.  A tangle diagram can be \emph{thickened,} or turned into a three dimensional object as follows. Embed the domain in the three sphere $S^3.$ For each crossing, replace a small neighborhood of the crossing with two strands each connecting opposite edges of the crossing, with the original two edges labeled ``over'' connected by a strand going over the strand connecting the two edges labeled ``under'' (after making a global choice of a direction to be up).  The result is an embedding of a collection of intervals and loops in $S^3,$ with the boundary points of the intervals coinciding with the original boundary vertices.  

If the tangle was checkerboarded, then each of the crossings has two dark faces touching it.  When replacing the quadrivalent vertex with a literal crossing, the two dark faces can be joined into a single dark surface connecting the over and under strands.  Doing this over the entire checkerboarded tangle yields a surface with boundary (the boundary consists of all the embedded intervals and loops making up the three-dimensional tangle and all the dark portions of the boundary of the disk).  Isotopies of the domain correspond to isotopies of the dark surface.  If $\tilde{T}$ is a checkerboarded tangle let $e_{\tilde{T}}$ be the Euler number (Euler characteristic) of the dark surface.  Each crossing is labeled positive or negative depending on whether the dark surface rotates clockwise or counterclockwise as you pass through it (this is independent of the direction in which you pass through).  Define $w_{\tilde{T}}$ to be the number of positive crossings minus the number of negative crossings in $\tilde{T}$ (more invariantly, this is the winding number of the framing vector relative to the surface) [Figure~\ref{fig:checktngl}].

\subsection{Invariance mod $4$}
\begin{thm}\label{thm:mod4main}
In the skein module expansion of any tangle $T $ in terms of
basis tangles, for any particular basis tangle $ B$ in the
expansion, the coefficient contains only powers of A which are equal
$\!\!\pmod 4$. 
Choosing a checkerboarding of $T,$ the 
exponents 
 are congruent to $w_{\tilde{T}} +  2 e_{\tilde{T}} -  2e_{\tilde{B}} 
 \pmod 4.$
\end{thm}

\begin{proof}
The ring $\ZZ[A,A^{-1}]$ admits a $\ZZ/4\ZZ$ grading by the exponent of $A.$  That is, the grading of  a monomial $nA^k$ is equal to $k \bmod 4,$ and the grading of a product of monomials is the sum of their gradings.  Thus the ring can be written (as a $\ZZ$ module) as a sum of its grade $0,1,2,3$ parts.  Once a checkerboarding is chosen for $T$ the module $\mathcal{T}_T$ admits a $\ZZ/4\ZZ$ grading which sends $nA^k \tilde{T}$ to $k+w_{\tilde{T}} +  2 e_{\tilde{T}} \pmod 4,$ such that when you multiply a monomial times such a generator the gradings add.  Now observe that each of the generators (Eqs.~\eqref{eq:kauffman1} and~\eqref{eq:kauffman2}) of $\mathcal{K}_{B}$ is a linear combination of terms of the same grade.
 Writing $\bracket{T}$ as a sum of monomials times base tangles, the sum of terms of grade different from $T$ is an element of the Kauffman bracket and therefore is zero.

\end{proof}

\begin{cor}
The span of a single coefficient of a base tangle in the expansion of $T$ is a non-negative multiple of 4.
\end{cor}

\section{Span of Coefficients}

\begin{defn}
The \textbf{positive Kauffman Bracket Polynomial} (pKBP) is the unique assignment of a polynomial in $\ZZ[A,A^{-1}]$ 
to each framed tangle diagram $T,$ denoted $[T],$ satisfying Eqs.~\eqref{eq:pkauffman1} 
and~\eqref{eq:pkauffman2}  and sending the standardly embedded unknot to $1.$  Note this differs from the Kauffman bracket polynomial only in the sign of the first equation.  The pKBP is not a knot invariant. 
\begin{equation}\label{eq:pkauffman1}
\left[ \inimg{0T1T.png}{1em}\right]=(A^2+A^{-2})\left[ \inimg{0T0T.png}{1em}\right]
\end{equation}
\begin{equation} \label{eq:pkauffman2}
\left[ \inimg{2T1pT.png}{1em}\right]=A\left[ \inimg{2B2T.png}{1em}\right]+A^{-1}\left[ \inimg{2B1T.png}{1em}\right]
\end{equation}
\end{defn}

The advantage of the pKBP from the current point of view is that, because the
coefficients of the skein relations are all positive, there is no
cancellation, and any term appearing in the calculation adds to the final
result. 

The skein module for the pKBP is the free module of planar isotopy
classes of tangle {diagrams modulo the pKBP relation.  As before
$[T]$ will be a linear combination of base tangles of the same type,
with coefficients in $\ZZ[A,A^{-1}]$ (now with all nonnegative
coefficients).  

\begin{defn}
Let $T$ be a tangle with decomposition into base tangles $\bracket{T}= \sum_B P^{T}_B(A){ B}.$ Define the \textbf{total span} of $T$ to
  be the highest power occurring in any $P^T_B$ minus the lowest
  power.  Define the \textbf{span} of $T$ to be the maximum of the
  span of each $P^T_B.$
\end{defn}

\begin{defn}
Let $T$ be tangle with decomposition into base tangles $[T]= \sum_B \overline{P}^{T}_B(A){ B}.$ Define the \textbf{total pKBP span} of $T$ to
  be the highest power occurring in any $\overline{P}^T_B$ minus the lowest
  power.  Define the \textbf{pKBP span} of $T$ to be the maximum of the
  span of each $\overline{P}^T_B.$ The span and total span are
  bounded by the pKBP span and total pKBP span respectively.
\end{defn}

\begin{defn}
For a tangle diagram $T$ define $i_T$ to be the number of interior
faces of $T,$ $n_T$ to be the number of crossings of $T,$ $g_T$ to
be the number of boundary points of $T,$ $c_T$ to be the number of
connected components of $T,$ and $c'_T$ to be the number of connected
components of $T$  which do not touch the boundary (see Figure~\ref{quanttngl}).  
\end{defn}

\begin{figure}[!ht]
\centering
  \includegraphics[width=0.4 \textwidth]{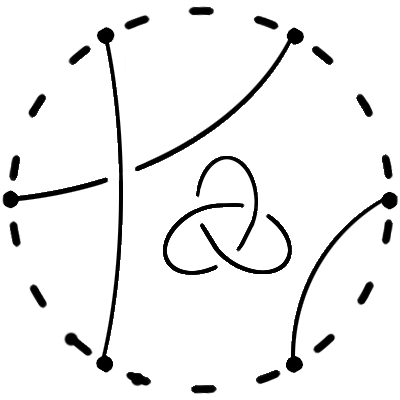}
 
  \caption{A tangle diagram $T$ with $i_T=4,$ $n+T=4,$ $g_T=6,$ $c_T=3,$ and $c'_T=1.$ }\label{quanttngl}
\end{figure}

\begin{prop}\label{totalspan}
The total pKBP span of  a link diagram $T$ 
is bounded by $4(n_T+c_T).$
\end{prop}
\begin{proof}
The proof is exactly the same as the proof of the analogous result for
the KBP, which is Theorem 1 of \cite{Murasugi87}.  The proof appearing
in that work only uses the defining relations of the polynomial in
Eq. (15) and the passage from there to Eq. (16), where the sign of the
constant has no bearing. 
\end{proof}

\begin{defn} \label{def:star}
If $T$ is a tangle and $B$ is a base tangle of the same type, define
$B*T$ to be the link built as follows.  Reflect $B$ about the circle
boundary so it is a graph outside the disk that touches the circle at
the same points as $T$ does.  Glue the two graphs together to get a
closed graph in a larger disk, and replace each of the now bivalent
vertices on the circle with a single continuous edge.  Observe that
$B*B$ will be a union of $g/2$ loops, if $B$ has $g$ boundary points,
and if $B'$ is any base tangle of the same type as $B,$ $B*B'$
consists of $g/2$ or fewer loops.
\end{defn}

\begin{figure}[!ht]
\centering
  \includegraphics[width=0.4 \textwidth]{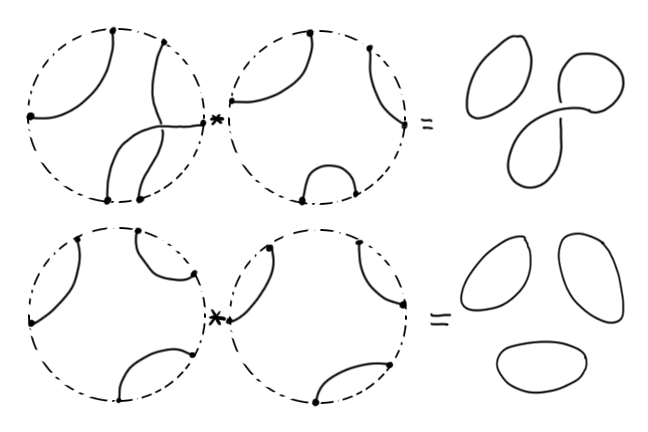}
 
  \caption{An example of $T*B$ and $B*B$ }\label{girth}
\end{figure}

In fact it is not too difficult to show, following Murasagi's proof,
that the total pKBP span of a tangle is bounded by $2(i_T+n_T+c'_T).$  However, the
above result, explored with the use of $B*T,$ gives a
stronger bound on the ordinary span, as needed here.

\begin{prop}
The span of $[T],$ and therefore the span of $\bracket{T},$ is bounded by
$4i_T= 4(n_T +c_T)-2g_T.$ 
\end{prop}
\begin{proof}
Let $B$ be a  base tangle whose coefficient $\overline{P}^T_B$ in the expansion of $[T]$
has the highest span.  Then by locality
\[[B*T]= \sum_{B'} \overline{P}^T_{B'}(A) [B*B']\]
and therefore the total span of $[B*T]$ is at least the span of $T$ plus
$2g,$ the $2g$ coming from the span of $[B*B]$ (this argument uses the fact that it is the positive
Kauffman bracket for the first time, because in the ordinary Kauffman
bracket terms could cancel in the product).  On the other hand the
total span of $[B*T]$ is, by Proposition~\ref{totalspan}, bounded by
\[4(n_{B*T}+c_{B*T}).\]
But $n_{B*T}=n_T,$ and notice that any $B$ which occurs in the
expansion of $T$ connects two boundary points only if they bound the
same connected component of $t,$ so $c_{B*T}=c_T.$  It thus follows
that the pKBP span of $T$ is bounded by 
\[4(n_T+c_T)-2g_T.\]
The equation $i_T=n_T+c_T-g_T/2$ can be proven by computing the  Euler number
or by induction on the
complexity of $T$ (each internal face which is not bounded by a simple closed
curve can be removed by smoothing one crossing that it bounds without
changing $c_T.$  Each internal face that is bounded by a simple closed curve
can be removed, reducing $c_T$ by $1.$  Each crossing not bounding an
internal face can be removed, increasing $c_T$ by $1.$)
\end{proof}

The results of the previous two sections give a bound on the amount of
information contained in the Kauffman bracket skein module expansion
of a tangle in terms of the crossing number and the type of its
boundary.  First notice that the number of basis tangle generating
$\mathcal{M}_B$ if $B$ has $g$ boundary points is the number of ways
of connecting $g$ points on the circle by paths on the disk so that no
two paths cross.  This quantity is the \emph{Catalan number,} 
$C_{m}$ for $m=g/2,$ which is given by the formula \cite{Koshy09}
\[\frac{1}{m+1} {\binom{2m}{m}}\]
and asymptotically
\[C_m \sim \frac{4^m}{m^{3/2} \sqrt{\pi}}.\]

\begin{prop}\label{compute}
If $T$ is a tangle with $g$ boundary points and a diagram with $n$
crossings, its image $\bracket{T}$ in $\mathcal{M}_B$ can be stored with
$(n-g/2+1)C_{g/2}$ integers, or asymptotically $n2^g.$  
\end{prop}

\begin{proof}
The expansion of $T$ will be a linear combination of $C_{g/2}$ basis
tangles.  The coefficient of each will be an element of $\ZZ[A,
A^{-1}].$  This coefficient is a sum of powers of $A,$ with all
nonzero terms having exponents that agree $\bmod 4,$ and the span of
exponents being $4n-2g.$  Such a polynomial is described by an integer
representing the highest exponent that occurs and an integer
representing each of the $n-g/2$ nonzero coefficients of powers of $A.$
\end{proof}

\section{Girth}

\begin {defn}
If $L$ is the diagram of a link, define a \textbf{cutting} of $L$ to be
a sequence of tangles $T_0 \subset T_1 \subset \cdots \subset T_n$
where $T_0$ is empty, $T_n$ is $L$ (more precisely, $L$ embedded in a
disk), each component of the graph $T_i-T_{i-1}$ is an edge connecting
a boundary vertex of $T_{i-1}$ to $T_{i}$ except one, which contains at most
one crossing.  Intuitively a cutting builds $L$ up one crossing at a
time.   It is easy to see that such a cutting exists for any $L$
(choose a generic height function, slide cups up and caps down as far
as they will go, and choose generic heights for the tangle
boundaries).  The \textbf{girth of a cutting} is the maximum number
of boundary vertices of all of these tangles, and the \textbf{girth of
  $L$} is the minimum girth over all cuttings.  The girth of a link is
the minimum girth of all diagrams.  See Figure~\ref{fg:girth} for a
cutting of the figure eight knot of girth $4$ (only the part of the
tangle boundary that intersects $L$ are shown).
\end{defn}

\begin{figure}[!ht]
\centering
  \includegraphics[width=0.4 \textwidth]{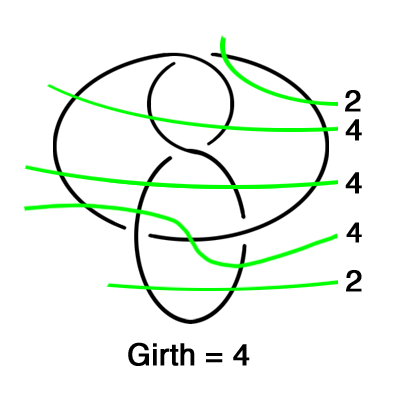}
 
  \caption{A cutting of $L$ with girth $4$ }\label{fg:girth}
\end{figure}

\begin{prop}
If $L$ is a  connected link projection with $n$ crossings and girth
$g$ then $\bracket{L}$ can be computed with $O(n2^g)$ memory and in time
$O( 2^g)$ times a polynomial in $n$ and $g.$
\end{prop}

\begin{proof}
Let $T_1,T_2,\ldots , T_n$ be a cutting of minimal girth.  We will compute $\bracket{L}$ by computing $\bracket{T_i}$
in sequence.  Because each $T_i$ has $g$ or fewer boundary points and
$n$ or fewer crossings, $\bracket{T_i}$ can be stored in $O(n2^{g/2}),$ units
of memory, and each $\bracket{T_i}$ can be discarded once $\bracket{T_{i+1}}$ is
computed.  To compute $T_{i+1}$ from $\bracket{T_i},$ it is necessary, for each of the
$C_{g/2}$ basis tangles, to do a finite set of calculations on each integer
coefficient using Eqs.~\eqref{eq:kauffman1} and~\eqref{eq:kauffman2}.  Since
there are $n$ such steps, the result follows.

\end{proof}

This naturally raises the question of what the girth of a link is, and
in particular if it can be bounded in terms of the number of crossings
$n.$  Indeed it can.  In particular the girth of a link is what is
known as the \emph{cutwidth} of the underlying graph embedded in the
plane.  In \cite{DV03}, Djidjev and Vrto show that the cutwidth of a
planar graph with $n$ vertices all of valence $4$ has cutwidth bounded
by 
\[\mathrm{cw} \leq (6\sqrt{2}+ 5 \sqrt{3}) \sqrt{n},\]
abbreviated $C\sqrt{n}.$

\begin{remark}  In fact it is not quite true that cutwidth as defined
  in \cite{DV03} is the same as our notion of girth for links.  Our
  definition implicitly requires each $T_i$ to live in a disk, whereas
  their definition allows subgraphs on regions which are not
  connected or simply connected, complicating the analysis.  However,
  \cite{DV03} defines the cutting recursively by using Theorem
  1.2 of Gazit and Miller \cite{GM90} that shows a planar graph can be
  cut into two substantial pieces with the number of cuts bounded by
  the square root of the number of vertices. In fact the argument in
  this paper constructing these two pieces produces disk-shaped
  regions, so the algorithm in \cite{DV03} in fact produces a cutting
  in our sense.
\end{remark}

\begin{thm}
The time and memory required to compute the Kauffman bracket of any
link of crossing number $n$ is
$\mathcal{O}\parens{\mathrm{poly}(n)2^{C\sqrt{n}}}$ for some $C>0.$  
\end{thm}

\begin{proof}
This follows from the bound above and Proposition~\ref{compute}.
\end{proof}

\bibliographystyle{elsarticle-num}
\def\cprime{$'$}

\end{document}